\documentclass[10pt]{article}
\usepackage{latexsym,amsmath,amsthm,amssymb}
\usepackage[sorted]{amsrefs}
\usepackage{graphicx}
\usepackage[colorlinks=true, linkcolor=blue]{hyperref}
\usepackage{geometry}
\usepackage{enumitem}
\usepackage{tikz}
\usepackage{mathtools}

\newtheorem{theorem}{Theorem}[section]
\newtheorem*{theorem*}{Theorem}

\newtheorem{lemma}[theorem]{Lemma}
\newtheorem{conjecture}[theorem]{Conjecture}

\newtheorem{claim}[theorem]{Claim}
\newtheorem{corollary}[theorem]{Corollary}
\newtheorem*{example*}{Example}
\newtheorem{example}[theorem]{Example}
\newtheorem{question}[theorem]{Question}
\newtheorem{obs}[theorem]{Observation}
\newtheorem{construction}{Construction}

\newtheorem*{definition*}{Definition}

\DeclareMathOperator{\psr}{psr}
\DeclareMathOperator{\sat}{sat}

\setlist[enumerate]{topsep=0pt,partopsep=1ex,parsep=1ex}
\usepackage{dsfont}

\newcommand{\abs}[1]{\left\lvert{#1}\right\rvert}

\title{Saturated Partial Embeddings of Planar Graphs}

\author{Alexander Clifton\thanks{Discrete Mathematics Group, Institute for Basic Science, Daejeon, South Korea {\tt yoa@ibs.re.kr}}\and Nika Salia\thanks{King Fahd University of Petroleum and Minerals, Dhahran, Saudi Arabia {\tt salianika@gmail.com}}}

\begin{document}
\maketitle

\begin{abstract}
In this work, we study how far one can deviate from optimal behavior when embedding a planar graph. 
For a planar graph $G$, we say that a plane subgraph $H\subseteq G$ is a \textit{plane-saturated subgraph} if adding any edge (possibly with new vertices) to $H$ would either violate planarity or make the resulting graph no longer a subgraph of $G$. 
For a planar graph $G$, we define the \textit{plane-saturation ratio}, $\psr(G)$, as the minimum value of $\frac{e(H)}{e(G)}$ for a plane-saturated subgraph $H$ of $G$ and investigate how small $\psr(G)$ can be. 
While there exist planar graphs where $\psr(G)$ is arbitrarily close to $0$, we show that for all twin-free planar graphs, $\psr(G)>1/16$, and that there exist twin-free planar graphs where $\psr(G)$ is arbitrarily close to $1/16$. In fact, we study a broader category of planar graphs, focusing on classes characterized by a bounded number of degree $1$ and degree $2$ twin vertices. We offer solutions for some instances of bounds while positing conjectures for the remaining ones.
\end{abstract}

\section{Introduction}

One prominent class of problems in extremal combinatorics concerns the \emph{Turán number}~\cite{turan1941external}. Given an integer $n$ and a graph $H$, the aim is to determine the maximum number of edges, $ex(n,H)$, in a graph with $n$ vertices, containing no subgraph isomorphic to $H$ (refer to the survey by Füredi and Simonovits~\cite{furedi2013history}). At the other end of the spectrum is the \emph{saturation number} $\sat(n,H)$, denoting the minimum number of edges in a $n$-vertex graph where there is no subgraph isomorphic to $H$, yet adding a further edge introduces a copy of $H$. 
Zykov~\cite{zykov1949some} initiated the study of the saturation problem in 1949, while in 1964, Erdős, Hajnal, and Moon~\cite{erdos1964problem} determined the exact value for $\sat(n,K_t)$. 
Bollobás~\cite{bollobas1965generalized} extended this result to hypergraphs in 1965.  Kászonyi and Tuza~\cite{kaszonyi1986saturated} established a general upper bound for $\sat(n,H)$ in 1986, which was later improved by Faudree and Gould~\cite{FG13} in 2013. In 2022, Cameron and Puleo~\cite{CP22} gave the first general lower bound. Further developments are outlined in the survey by Faudree, Faudree, and Schmitt~\cite{faudree2011survey}.

In this paper, we study a variation of saturation problems concerning planar drawings of planar graphs. 
The inspiration behind this problem comes from the work of Kynčl, Pach, Radoičić, and Tóth~\cite{KPRT15}. 
They investigate simple topological graphs $G$, defined as graphs drawn in the plane such that any pair of edges shares at most one common point. 
A topological graph $G$ is deemed saturated if no additional edge can be added into $G$ without compromising its topological nature (see also \cites{pach2003disjoint,brass2005research,pach2003unavoidable,suk2012disjoint,fulek2013topological}).


Every \emph{planar graph} has a planar embedding, i.e. there exists an embedding on the plane such that no two edges intersect (or overlap) and no two vertices coincide.  
A planar graph with a given planar embedding is called a \emph{plane graph}.
In fact, if we are not judicious with how we embed the edges and vertices, we may reach a point where adding any additional edge of a planar graph to the drawing will cause a crossing.
To demonstrate this phenomenon, we present the following example.

\begin{example}\label{example0}
For an integer $n>5$, let $G_n$ be an $n$-vertex planar graph consisting of a triangle $v_1v_2v_3$, a vertex disjoint complete bipartite graph $K_{2,n-5}$ with one of the partite classes $\{u_1,u_2\}$, and edges $u_1v_1, u_2v_2$.

Let $H_n$ be an $n$-vertex plane graph, consisting of a triangle $w_1w_2w_3$, two vertices $w'$ and $w''$ in the interior of the triangle such that $w'$ is adjacent to $w_1$ and $w''$ is adjacent to $w_2$, and $n-5$ isolated vertices embedded in the exterior of the triangle. 

The plane graph $H_n$ is a subgraph of $G_n$ and it is apparent that the vertices $w',w''$ are representatives of $u_1,u_2$, since there is a unique triangle in $G_n$. Hence, adding any additional edge of $G_n$ to $H_n$ will cause a crossing. 

As the number of edges in $H_n$ is a constant, independent from $n$, the ratio $\frac{e(H_n)}{e(G_n)}$
tends to $0$ as $n$ tends to infinity.
\end{example}

In this paper, we study a natural question of how inefficient one can be while embedding a planar graph on the plane. To this end, we introduce some definitions necessary for a rigorous statement of the problem.

\begin{definition*}
A plane graph $H$ is a \textit{plane-saturated subgraph} of a planar graph $G$ if, when any edge (potentially introducing new vertices) is added to $H$, it either introduces a crossing or causes the resulting graph to no longer be a subgraph of $G$.



For a graph $G$, we define the \textit{plane-saturation ratio} $\psr(G)$ as the minimum value of $\frac{e(H)}{e(G)}$ over all plane-saturated subgraphs $H$ of $G$.
\end{definition*}

For example, the plane-saturation ratio of a tree or a cycle is one.

Note that to determine the plane-saturation ratio of a graph, it is enough to consider all plane-saturated subgraphs $H$ of $G$ with $v(H)=v(G)$, since adding isolated vertices to $H$ will not affect whether $H$ is a plane-saturated subgraph of $G$ or not. We observe that removing isolated vertices from $G$ does not affect $\psr(G)$, so we may also assume throughout the paper that $G$ has no isolated vertices.

The plane-saturation ratio can be thought of as a measure of how far one can deviate from optimal behavior when drawing a given planar graph on the plane while forbidding crossings. It is natural to ask for the lowest possible value of $\psr(G)$ for a planar graph $G$. 
However, 
if $G_n$ is the planar graph described in Example \ref{example0}, then $\psr(G_n)\leq \frac{5}{2n-5}$ and $\lim_{n\to\infty}\psr(G_n)=0$. 
The presence of arbitrarily many degree $2$ vertices with the same neighborhood in $G_n$  allows for $\psr(G)$ to be arbitrarily close to $0$. 
Similarly, removing $u_1$ and $w'$ from $G_n$ and $H_n$ respectively, along with their incident edges gives an example of a planar graph $G_{n-1}'$ and its plane-saturated subgraph demonstrating $\psr(G_{n-1}')\le \frac{4}{n-1}$. 
Thus it is possible for $\psr(G)$ to be arbitrarily close to $0$ due to the presence of arbitrarily many degree $1$ vertices with the same neighborhood. It is natural to wonder if $\psr(G)$ is bounded away from $0$ if we restrict to planar graphs that do not have arbitrarily many degree $1$ vertices or degree $2$ vertices with the same neighborhood.

 For constants $k_1$ and $k_2$, we study the infimum value of $\psr(G)$ for $G$ in the class of planar graphs with at most $k_1$ degree $1$ vertices with the same neighborhood and at most $k_2$ degree $2$ vertices with the same neighborhood. Of particular interest is the case $k_1=k_2=1$, meaning $G$ has no twins\footnote{A pair of vertices are called twins if they share the same neighborhood.} of degree $1$ or $2$.
 
 In the following example, we demonstrate a twin-free planar graph with one of its plane-saturated subgraphs. 

\begin{example}\label{ex:onefifth}
  Let $G_{2n+5}$ be a twin-free planar graph on $2n+5$ vertices consisting of a matching of size $n$, two vertices $u_1$ and $u_2$ adjacent to every vertex of the matching, and a triangle $v_1v_2v_3$, with additional edges $v_1u_1$ and $v_2u_2$; see the left graph in Figure~\ref{fig:partofH}. 
  
  Let $H_{2n+5}$ be the plane graph formed by $G_{7}$ and a matching of size $n-1$ embedded in the plane as in the right graph in Figure~\ref{fig:partofH}. 
    
    \begin{figure}[ht]
        \centering
 \begin{tikzpicture}
  \tikzstyle{vertex}=[circle,fill=black, inner sep=1.2pt]  
  \node[vertex] (u_1) at (2.35,3.3) {};
  \node[vertex] (u_2) at (2.35,1.2)   {};

  \node[vertex] (a_1) at (6,3.4) {};
  \node[vertex] (a_2) at (4.5,1.1)   {};
  \node[vertex] (a_3) at (7.3,1.1)  {};
  
  \node[vertex] (a_6a) at (1,2.25)   {};
  \node[vertex] (a_6b) at (1.5,2.25)  {};
  \node[vertex] (a_7a) at (3,2.25)   {};
  \node[vertex] (a_7b) at (3.5,2.25)  {};

  \node[vertex] (a_9a) at (4,2.25)   {};
  \node[vertex] (a_9b) at (4.5,2.25)  {};
  \node[vertex] (a_0a) at (0,2.25)   {};
  \node[vertex] (a_0b) at (0.5,2.25)  {};

    \draw (a_1) -- (a_2)--(a_3)--(a_1);
    \draw (u_1) -- (a_1);
    \draw (u_2) -- (a_2);
  
    \draw (a_6a) -- (a_6b);
    \draw (a_7a) -- (a_7b);
    \draw (2.25,2.25)node[align=center]{$\dots$};
    \draw (a_9a) -- (a_9b);
    \draw (a_0a) -- (a_0b);
    
    \draw (u_1) -- (a_6a);
    \draw (u_1) -- (a_7a);
    \draw (u_1) -- (a_0a);
    \draw (u_1) -- (a_9a);
    \draw (u_1) -- (a_6b);
    \draw (u_1) -- (a_7b);
    \draw (u_1) -- (a_0b);
    \draw (u_1) -- (a_9b);

    \draw (u_2) -- (a_6a);
    \draw (u_2) -- (a_7a);
    \draw (u_2) -- (a_0a);    
    \draw (u_2) -- (a_9a);
    \draw (u_2) -- (a_6b);
    \draw (u_2) -- (a_7b);
    \draw (u_2) -- (a_0b);
    \draw (u_2) -- (a_9b);

 \draw [dashed] (7.5,1.1) -- (7.5,4.3);

\end{tikzpicture}
\begin{tikzpicture} 
  \tikzstyle{vertex}=[circle,fill=black, inner sep=1.2pt]
  \node[vertex] (a_1) at (6,3.4) {};
  \node[vertex] (a_2) at (4.5,1.1)   {};
  \node[vertex] (a_3) at (7.3,1.1)  {};
  \node[vertex] (a_4) at (6.5,1.75) {};
  \node[vertex] (a_5) at (5.5,1.75) {};
    \node[vertex] (a_6a) at (8,2.9)   {};
  \node[vertex] (a_6b) at (8,2)  {};
  \node[vertex] (a_7a) at (8.5,2.9)  {};
  \node[vertex] (a_7b) at (8.5,2)  {};
  \node[vertex] (a_10a) at (9,2.9)  {};
  \node[vertex] (a_10b) at (9,2)  {};
  \node[vertex] (a_11a) at (11,2.9)  {};
  \node[vertex] (a_11b) at (11,2)  {};
  \node[vertex] (a_12a) at (11.5,2.9)  {};
  \node[vertex] (a_12b) at (11.5,2)  {};
  \node[vertex] (a_8a) at (6,1.5){};
  \node[vertex] (a_8b) at (6,2){};
  \draw (a_1) -- (a_2)--(a_3)--(a_1);
  \draw (a_3)--(a_4);
  \draw (a_2)--(a_5);
  \draw (a_4)--(a_8a)--(a_5);
  \draw(a_4)--(a_8b)--(a_5);
  \draw (a_6a)--(a_6b);
  \draw (a_7a)--(a_7b);
  \draw (a_8a)--(a_8b);
  \draw (a_11a)--(a_11b);
  \draw (a_10a)--(a_10b);
  \draw (a_12a)--(a_12b);
\draw (10,2.5)node[align=center]{$\dots$};
\end{tikzpicture}

        \caption{On the left, graph $G_{2n+5}$ and on the right, graph $H_{2n+5}$ from Example \ref{ex:onefifth}.}
        \label{fig:partofH}
    \end{figure}
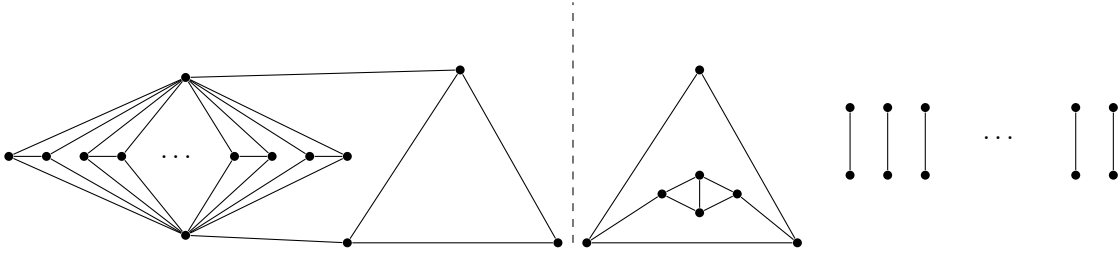
    First, we show that $H_{2n+5}$ is a plane-saturated subgraph of $G_{2n+5}$.
    It is easy to see that $H_{2n+5}$ is a subgraph of $G_{2n+5}$. Fix an embedding of $H_{2n+5}$ in $G$. 
    For $i\in \{1,2\}$, consider the induced graph $G[V(G_{2n+5})-\{u_i\}]$; it contains no two triangles sharing an edge. 
    Furthermore, every two triangles that share an edge $ab$ in $G_{2n+5}$ are $u_1ab$ and $u_2ab$. 
    As $H_{2n+5}$ contains two triangles sharing an edge, the pre-images of the vertices $u_1$ and $u_2$ are well defined up to swapping in $H_{2n+5}$.
    Moreover, $H_{2n+5}$ contains a triangle not incident to the pre-images of $u_1$ and $u_2$; such a triangle is also unique in $G_{2n+5}$. 
    Namely, it is  $v_1v_2v_3$. 
    Thus, the set of pre-images of $v_1,v_2,v_3$ in $H_{2n+5}$ is well defined. 
    Every other vertex in $G_{2n+5}$ has degree $3$ and two of its neighbors are $u_1$ and $u_2$. 
    However, since pre-images of the vertices $u_1$ and $u_2$ are separated from the remaining vertices in $H_{2n+5}$, no further edge may be added to $H_{2n+5}$ without compromising planarity.   
    

     Thus, we have 
     \[
     \lim_{n\to\infty}\psr(G_{2n+5})\le \lim_{n\to\infty}\frac{e(H_{2n+5})}{e(G_{2n+5})} =  \lim_{n\to\infty}\frac{n+9}{5n+5}=\frac{1}{5}.
     \]   
\end{example}

While it is reasonable to conjecture that as $n$ tends to infinity, the graph $G_{2n+5}$ from Example~\ref{ex:onefifth} achieves the minimum plane-saturation ratio for twin-free planar graphs, there exist such graphs with smaller plane-saturation ratio.

\begin{theorem}\label{thm:Twin-free}
    For every planar graph $G$ with no degree $1$ or degree $2$ twins,
    \[
        \psr(G)>\frac{1}{16}.
    \]
    For every positive $\epsilon$, there exists a twin-free planar $G_{\varepsilon}$ such that 
    \[
   \psr(G_{\epsilon})<\frac{1}{16}+\epsilon.
    \]
\end{theorem}

\begin{corollary}\label{cor:actualtwinfree}
    Any twin-free planar graph $G$ satisfies $\psr(G)>1/16$.
\end{corollary}

In this work, we prove the following more general theorem which directly implies $\psr(G)>\frac{1}{16}$ in the first part of Theorem~\ref{thm:Twin-free}. 
Let $\mathcal{G}_{k_1,k_2}$ denote the class of planar graphs where at most $k_1$ vertices of degree $1$ have the same neighborhood and at most $k_2$ vertices of degree $2$ have the same neighborhood.

\begin{theorem}\label{thm:mostK}
    For a positive integer $k_1$ and a non-negative integer $k_2$ with $(k_1,k_2)\ne (1,0), (2,0)$,
    every planar graph $G\in\mathcal{G}_{k_1,k_2}$satisfies
    \[
    \psr(G)>\frac{1}{9+k_1+6k_2}.
    \]
    Furthermore, for all non-negative integers $k_1$ and $k_2$ and every positive $\epsilon$, there exists a graph $G_{\varepsilon}\in \mathcal{G}_{k_1,k_2}$ such that
    \[
   \psr(G_{\epsilon})<\frac{1}{9+k_1+6k_2}+\epsilon.
    \]
\end{theorem}


We follow standard notations;  for a graph $G$, the set of its vertices is denoted by $V(G)$ and the set of its edges by $E(G)$. The number of vertices and edges are denoted by $v(G)$ and $e(G)$, respectively. If $G$ is a planar graph, then $f(G)$ denotes the number of faces.
For a vertex $v$ of $G$, the neighborhood of $v$ is the set of all vertices $u$ such that $vu$ is an edge of $G$ and is denoted by $N_G(v)$. 
The degree of the vertex $v$ is the size of $N_G(v)$ and is denoted by $d_G(v)$.
When the host graph is apparent from the context, we may use $N(v)$ and $d(v)$ instead of $N_G(v)$ and $d_G(v)$, respectively.  
When $G$ is a graph and $e$ is a pair of its nonadjacent vertices, then 
$G+e$ denotes the graph obtained from $G$ by adding the edge $e$.
When $U$ is a subset of vertices of $G$, then $G[U]$ denotes the induced subgraph of $G$ with vertex set $U$ and we denote graph $G[V(G)- U]$ by $G-U$. The cycle graph with $n$ vertices is denoted by $C_n$.

We introduce some notions necessary for the rest of the paper. 
Let $G$ be a planar graph and $H$ be a plane-saturated subgraph of $G$. 
Let $I$ denote the set of isolated vertices of $H$. 
We refer to the graph $H-I$ as the \textit{skeleton} and consequently the set of vertices $V(H)- I$ as \textit{skeleton vertices}. 
In particular, the skeleton is a planar graph without a specified embedding. 
We call a skeleton component with exactly two vertices a \textit{matching edge}. We will sometimes refer to an \textit{embedding} $\phi:V(H)\rightarrow V(G)$ of $H$ in $G$ where $\phi$ is a bijective function such that for every edge $uv$ of $H$, $\phi(u)\phi(v)$ is an edge of $G$. With a slight abuse of notation, for every edge $uv$ of $H$ we denote $\phi(uv):=\phi(u)\phi(v)$.

A strategic choice was made to prioritize the proof of Theorem~\ref{thm:Twin-free} as an initial step, even though the first part can be viewed as a special case of Theorem~\ref{thm:mostK}. 
This decision is based on the technicality involved in the proof process, coupled with our belief that Theorem~\ref{thm:Twin-free} represents the most foundational and natural case within the broader context of Theorem~\ref{thm:mostK}. 
The rest of the paper has the following structure.
Section \ref{section:twinfree} is dedicated to the proof of Theorem~\ref{thm:Twin-free}. 
In Subsection \ref{sub:upperbound}, we construct an infinite family of twin-free planar graphs that obtain plane-saturation ratios arbitrarily close to $\frac{1}{16}$. 
In Subsections \ref{sub:structure} and \ref{sub:counts}, we make some observations about the structure of a planar graph $G$ which has a plane-saturated subgraph $H$ with a given skeleton and use this to bound the number of edges of $G$ and $H$ in terms of the numbers of each type of component in the skeleton of $H$. 
In Subsections \ref{sub:proof1} and \ref{sub:proof2}, we complete the proof by considering two cases and showing $\frac{e(H)}{e(G)}>1/16$ for each. 
In Section \ref{section:mostK}, we indicate how the proof of Theorem~\ref{thm:Twin-free} can be modified to give a proof of Theorem~\ref{thm:mostK}. 
Finally in Section \ref{section:conc}, we consider some related open questions.

\section{Proof of Theorem~\ref{thm:Twin-free}}\label{section:twinfree}
\subsection{Construction of a twin-free graph achieving the desired upper bound}\label{sub:upperbound}

In this section, we present a family of twin-free planar graphs for which the plane-saturation ratio $\psr(G)$ tends to $\frac{1}{16}$, settling the second part of Theorem~\ref{thm:Twin-free}.

\begin{construction}\label{decorateddoublewheel}
Let $m\geq{7}$ and $G_0$ be a graph on $m+2$ vertices consisting of a $C_m$ and two additional vertices which are each adjacent to every vertex of the $C_m$.
There exists a unique planar embedding of $G_0$. 
Note that $G_0$ has $m+2$ vertices, $3m$ edges and $2m$ faces.

Let $G_1$ be the planar graph obtained from $G_0$ in the following way. 
For each face $f$ of $G_0$, we add a vertex $u_f$, of degree $3$, adjacent to the three boundary vertices of $f$. 
For each edge $e$ of $G_0$, we add a vertex $u_e$, of degree $2$ which is adjacent to the two endpoints of $e$. 
For each vertex $v$ of $G_0$, we add a vertex $u_v$, of degree $1$, which is adjacent to $v$.

Let $G_{3/m}$ be the disjoint union of $G_1$ and $K_4$. Note that $G_{3/m}$ is planar and free of degree $1$ and degree $2$ twins. Furthermore, since every vertex of $G_0$ has a pendant edge in $G_{3/m}$ and every additional vertex has a different set of neighbors,  $G_{3/m}$ is in fact twin-free. 
 Thus by the following Claim~\ref{ubworks}, we have that for every positive $\epsilon$, there exists a twin-free planar $G_{\varepsilon}$ such that 
    \[
    \psr(G_{\epsilon})<\frac{1}{16}+\epsilon,
    \]
and we have completed the proof of the second part of Theorem~\ref{thm:Twin-free}.

\end{construction}

\begin{claim}\label{ubworks}
  Let $G_{3/m}$ be as it is defined in Construction \ref{decorateddoublewheel}. Then
  \[
  \psr(G_{3/m})< \frac{1}{16}+\frac{3}{m}.
  \]
\end{claim}

\begin{proof}

The  number of vertices of $G_{3/m}$ is
\[
v(G_{3/m})=(m+2)+2m+3m+(m+2)+4=7m+8.
\]
 The number of edges in $G_{3/m}$ is 
\[
   e(G_{3/m})= \left(3m+3(2m)+2(3m)+(m+2)\right)+6=16m+8.
\]

Consider a planar graph $H_0$ consisting of the disjoint union of $C_m$, $K_4$, two copies of $K_{1,14}$, and $6m-26$ isolated vertices. 
The number of vertices of $H_0$ is $v(H_0)=m+4+2(15)+(6m-26)=7m+8=v(G_{3/m})$ and $e(H_0)=m+6+2\cdot14=m+34$.

It is easy to see that $H_0$ is a subgraph of $G_{3/m}$ as the two vertices of $G_{3/m}$ adjacent to every vertex on the initial $C_m$ both have at least fourteen neighbors that are not on the cycle and these neighborhoods are disjoint. 
Now let $H$ denote the planar embedding of $H_0$ as follows: 
We first draw the $K_4$ without crossings so it has four triangular faces. 
Inside one of the faces, we embed $C_m$, inside another we embed two $K_{1,14}$'s, and inside a third, we embed $6m-26$ isolated vertices.

We now demonstrate that it is impossible to add any edge $e$ to $H$ without destroying planarity
such that the resulting plane graph is still a subgraph of $G_{3/m}$.
In $G_{3/m}$, every vertex has a degree at most $3$, except the vertices $V(G_0)$: the $C_m$ and the two high-degree vertices adjacent to all vertices of that $C_m$. 
A vertex from  $V(C_m)$ has degree $13$ in $G_{3/m}$. Indeed, it is adjacent to two other vertices on the $C_m$, the two high-degree vertices, four additional vertices of degree $3$ (corresponding to incident faces in $G_0$), four additional vertices of degree $2$ (corresponding to incident edges in $G_0)$, and one vertex of degree $1$. 
Each of the two high-degree vertices of $G_0$  has degree $3m+1\geq{22}$ in $G_{3/m}$. 
Indeed, each is adjacent to $m$ vertices in $G_0$, as well as $m$  degree $3$ vertices, $m$ degree $2$ vertices, and one degree $1$ vertex in $G_{3/m}$.

As $G_{3/m}$ only has two vertices of degree at least $14$, the two degree $14$ vertices in $H$ can be uniquely identified (up to swapping) in the original graph $G_{3/m}$. Now we seek to identify the $C_m$ which does not use either of these two vertices. 
If we look at the original graph $G_{3/m}$ and remove all edges incident to the two highest degree vertices which have already been identified, we are left with a $K_4$, which cannot contain a $C_m$, and a graph $G'$ as follows:

\begin{itemize}
    \item $G'$ contains a $C_m$,
    \item For each vertex $v$ of the $C_m$, there is a vertex of degree $1$ adjacent to $v$,
    \item For each edge $uv$ of the $C_m$, there are three vertices of degree $2$ with the neighborhood $\{u,v\}$. (In $G_{3/m}$, the corresponding vertices are the vertex with neighborhood $\{u,v\}$ and the two vertices adjacent to $u, v,$ and  one of the two high-degree vertices.)
\end{itemize}

We claim that $G'$ has only one copy of $C_m$. Suppose the initial $C_m$ is labeled $v_1,v_2,\dots,v_m,v_1$ and that the three degree $2$ vertices with neighborhood $\{v_i,v_{i+1}\}\pmod{m}$ are labeled $a_i, b_i$, and $c_i$. The degree $1$ vertices are not part of any cycle so we ignore them. 

If there is another $C_m$ besides $v_1,v_2,\dots,v_m,v_1$, then it contains either $a_i, b_i$ or $c_i$ for some $i$. For a given $i$, any cycle containing two of $\{a_i,b_i,c_i\}$ is simply a $C_4$ containing those vertices along with $v_i$ and $v_{i+1}$. As $m\geq{7}$, this is not a $C_m$ and it suffices to ignore every $b_i$ and $c_i$. Now if any $a_i$ is used in a cycle, then its neighbors $v_i$ and $v_{i+1}$ are both used, but the edge $v_iv_{i+1}$ is not, as then the cycle would be $C_3$. Thus, the rest of the cycle is a path from $v_{i+1}$ to $v_{i}$, which necessarily contains each $v_j$ for $1\leq{j}\leq{m}$. Thus, the total length of this cycle is at least $2+(m-1)>m$, a contradiction. 

As there is only one copy of $C_m$ in $G'$, the $C_m$ in $H$ can be uniquely identified (up to the rotation and reflection) in the original graph $G_{3/m}$. 
With the exception of the six edges in the $K_4$, every other edge of $G_{3/m}$ is incident to a vertex from $V(G_0)$. Therefore, the $K_4$ in $H$ can be identified as the disjoint $K_4$ in $G_{3/m}$.

The only potential edges we can add to $H$ without causing a crossing are:
\begin{itemize}
    \item[i)] Edges between a vertex of the identified $K_4$ and another vertex,
    \item[ii)] Additional edges between two vertices of the identified $C_m$,
    \item[iii)] Additional edges among the vertices of the two $K_{1,14}$'s,
    \item[iv)] Edges between two of the $6m-26$ isolated vertices.
\end{itemize}

It is easy to note that the edges of Types i), ii) do not exist in $G_{3/m}$. Edges of Type iii) also do not exist in $G_{3/m}$ since the two degree $3m+1$ vertices of $G_{3/m}$ are nonadjacent and have no common neighbors besides the already identified $C_m$. 
Furthermore, any edge of $G_{3/m}$ (besides the $K_4$ component) contains at least one vertex of $G_0$ so edges between the leaves of the two $K_{1,14}$'s do not exist in $G_{3/m}$, nor do edges of Type iv).
This establishes that $H$ is a plane-saturated subgraph of $G_{3/m}$. We observe that
\[
\frac{e(H)}{e(G_{3/m})}=\frac{m+34}{16m+8}< \frac{1}{16}+\frac{3}{m}.
\]
\end{proof}

\subsection{Lower Bound Proof}

In this section, we show that for every planar graph $G$ with no degree $1$ and degree $2$ twins, the following holds:
    \[
        \psr(G)>\frac{1}{16}.
    \]

\subsubsection{Preliminaries}\label{sub:structure}

Let $H$ be a plane-saturated subgraph of $G$.
Let $I$ denote the set of isolated vertices of $H$. Let us recall that we use \textit{skeleton} to refer to all the components of $H- I$ and \textit{skeleton vertices} to refer to $V(H)- I$. 
Let $H$ be the skeleton of $H$. 
We distinguish different types of connected components of the skeleton: 
\begin{itemize}
    \item connected components with a cycle,
    \item connected components that are trees of diameter at least three, 
    \item connected components isomorphic to stars with at least four vertices,
    \item connected components isomorphic to stars with three vertices,
    \item connected components isomorphic to edges (matching edges).
\end{itemize}

We will show that it is enough to assume that the image of $I$ is an independent set. 

\begin{lemma}\label{facewith2}
    If every face of  $H$ contains at most one isolated vertex, then    
    \[
       \frac{e(H)}{e(G)}>\frac{1}{16}.
    \]
\end{lemma}
\begin{proof}
Since $H'$ is planar, we have that $v(H')+f(H')=e(H')+k+1$ where $k$ is the number of connected components of $H'$. As each component contains an edge, we have that $k\le e(H')$, so $v(H')+f(H')\le 2e(H')+1$.
Since $|I|\le f(H')$, we have that $v(G)\le v(H')+f(H')\le 2e(H')+1=2e(H)+1$. Since $G$ is planar, we have $e(G)\le 3(2e(H)+1)-3=6e(H)$. 
Thus, \[\frac{e(H)}{e(G)}\ge 1/6>1/16,\] as desired.
\end{proof}

The following is a direct corollary of Lemma~\ref{facewith2}.

\begin{corollary}\label{Cor:independetnt}
    For every embedding of $H$ in $G$, the image of $I$ is an independent set or 
     \[
       \frac{e(H)}{e(G)}>\frac{1}{16}.
    \]
\end{corollary}
\begin{proof}
By  Lemma~\ref{facewith2}, we have either $\frac{e(H)}{e(G)}>\frac{1}{16},$
or there is a face of $H$ containing two isolated vertices.
If there is an embedding of $H$ in $G$ such that the image of $I$ induces an edge of $G$, then there is an embedding of $H$ in $G$ such that the images of the two isolated vertices from the same face induce an edge in $G$. Thus, it is possible to add this edge to $H$, a contradiction
since $H$ is a plane-saturated subgraph of $G$.
\end{proof}

\begin{obs}\label{obs:all_I_in_one_face}
By Corollary~\ref{Cor:independetnt}, we may assume that all vertices of $I$ are embedded in the same face of $H$ or
  \[
       \frac{e(H)}{e(G)}>\frac{1}{16},
    \] 
    and hence we are done.
\end{obs}

Choose an arbitrary embedding $\phi:V(H)\rightarrow V(G)$ of $H$ in $G$. 
Now we go through a series of steps to modify the embedding $\phi$ to increase the number of skeleton vertices mapped to vertices of degree $1$ or $2$.

\begin{enumerate}
\item \label{Phase_one} Phase 1
\begin{enumerate}
    \item Select a matching edge $e$ of the skeleton such that $\phi(e)=ab$, $\abs{N_G(b)\cap \phi(I)}\geq 2$, and there exists a vertex $c\in N_G(b)\cap \phi(I) $ with $d_{G}(c)=1$.
    
    \item We may instead embed this matching edge to the edge $bc$ of $G$, placing $a$ in the image of $I$, and not changing anything else.
    Thus we redefine  $\phi(\phi^{-1}(a)):=c$  and $\phi(\phi^{-1}(c)):=a$. 
    
    \item  We repeat this process until there are no more matching edges of this form. The process is guaranteed to terminate because, on each step, the number of degree one vertices in $\phi(I)$ is strictly decreasing since $G$ does not contain degree $1$ twins.
\end{enumerate}

\item \label{phase_two} Phase 2
\begin{enumerate}
\item  Select a matching edge $e$ of the skeleton such that $\phi(e)=ab$, $d_G(a)\geq 3$,  $\abs{N_G(b)\cap \phi(I)}\ge 2$, and there is a vertex $c\in N_G(b)\cap \phi(I)$ with $d_{G}(c)=2$. (All other $c'\in N_G(b)\cap \phi(I)$ satisfy $d_G(c')\ge 2$.)

\item We may instead embed this matching edge to edge $bc$ of $G$, placing $a$ in the image of $I$, and not changing anything else.
Thus, we redefine $ \phi(\phi^{-1}(a)):=c$ and $\phi(\phi^{-1}(c)):=a$. 

\item We repeat this process until there are no more matching edges of this form. 
The process is guaranteed to terminate because, on each step, the number of degree two vertices in $\phi(I)$ is strictly decreasing.
\end{enumerate}
\end{enumerate}

At the end of this process, we have now fixed our embedding $\phi$. For a matching edge $e$ of $H$, there are four possible ways by which  $\phi$ can embed $e$ in $G$.
\begin{enumerate}
    \item[Type M1:] A vertex of $\phi(e)$ has degree $1$ in $G$ and the other has at least two neighbors in $\phi(I)$, which both have degree at least $2$ in $G$.
    
    \item[Type M2:]  A vertex of $\phi(e)$ has degree $2$ in $G$ and the other has at least two neighbors in $\phi(I)$, which both have degree at least $2$ in $G$.
    
    \item[Type M3:] A vertex of $\phi(e)$ has at least two neighbors in $\phi(I)$, which both have degree at least $3$ in $G$, and the other vertex of $\phi(e)$ also has degree at least $3$ in $G$.
    
    \item[Type M4:] Both vertices of $\phi(e)$ have at most one neighbor in $\phi(I)$.
    
\end{enumerate}

Let $r_1, r_2, r_3,$ and $y$ denote the number of matching edges of Types M1, M2, M3, and M4 respectively.

We recall that $G$ is a planar graph with no degree $1$ and degree $2$ twins with a plane-saturated subgraph $H$ and that $\phi$ is an embedding of $H$ in $G$. 
The vertex set $I$ is the set of all isolated vertices of $H$, embedded in the same face of 
$H$ and $\phi(I)$ is an independent set of $G$ by Corollary~\ref{Cor:independetnt}. 
We now prove a series of results imposing restrictions on the edges of $G$ incident to one vertex in $\phi(I)$.

\begin{lemma}\label{lem:noedge}
The set of the degree 1 endpoints of the images of Type 1 matching edges and the degree 2 endpoints of the images of Type 2 matching edges is an independent set of $G$.
\end{lemma}
\begin{proof}
If there is such an edge, it must be $e=\phi(v_1)\phi(u_2)$ where $\phi(v_1)$, $\phi(u_2)$ have degree 2 in $G$ and $v_1$, $u_2$ are endpoints of Type M2 matching edges $e_1=v_1v_2$ and $e_2=u_1u_2$.
Then the vertices $\phi(v_2), \phi(u_1)$ each has at least two neighbors in $\phi(I)$, so without loss of generality, there exist distinct isolated vertices, $w_1, w_2$ in the same face of $H$ such that $\phi(w_1)\phi(v_2)$ and $\phi(w_2)\phi(u_1)$ are edges in $G$. Let $H_1$ be the plane graph obtained from $H$ by adding edge $w_1w_2$. It is easy to see that $H_1$ is a subgraph of $G$ with the embedding $\phi':V(H_1)\to V(G)$ defined as follows: $\phi'(u_2)=\phi(w_2),\phi'(w_2)=\phi(u_2),\phi'(v_1)=\phi(w_1),\phi'(w_1)=\phi(v_1)$, and $\phi'=\phi$ everywhere else.
One may add edge $w_1w_2$ without causing a crossing, so this contradicts the condition that $H$ is plane-saturated. Thus, there is no such edge $\phi(v_1)\phi(u_2)$.

    
\end{proof}

\begin{lemma}\label{no2separateNeighbors}
    Let  $\mathcal{C}$ be an acyclic connected component of the skeleton of $H$, $u$ be a leaf of $\mathcal{C}$,  $v$ be the sole neighbor of $u$ in $\mathcal{C}$ and $w_1, w_2\in I$. 
    If the edges $\phi(u) \phi(w_1)$ and $\phi(v)\phi(w_2)$ are edges of $G$, then $w_1=w_2$. 
\end{lemma}
\begin{proof}
By Lemma~\ref{facewith2}, there is a face of $H$ with two isolated vertices. Without loss of generality, we may assume $w_1$ and $w_2$ are in the same face of $H$.
Let $\phi':V(H)\rightarrow V(G)$ be a bijection obtained from $\phi$ in the following way:
$\phi'(u)=\phi(w_2)$, $\phi'(w_2)=\phi(u)$ and $\phi'=\phi$ everywhere else. 
If $w_1\neq w_2$, then let $H_1$ be the plane graph obtained from $H$ by adding edge $w_1w_2$. 
It is easy to see that $H_1$ is a subgraph of $G$ with the embedding $\phi'$.
Thus, $H$ is not a plane-saturated subgraph of $G$, since we may add an edge $w_1w_2$ without causing a crossing, a contradiction.
\end{proof}

\begin{corollary}\label{onlyshared}
     Let  $\mathcal{C}$ be an acyclic connected component of the skeleton of $H$, $u$ be a leaf of $\mathcal{C}$,  and $v$ be the sole neighbor of $u$ in $\mathcal{C}$. Then either,
    \begin{itemize}
        \item[i)] One of the vertices $\phi(u)$ and $\phi(v)$ has no neighbor in $\phi(I)$, 
        \item[ii)] or $N_G(\phi(u))\cap \phi(I)=N_G(\phi(v)) \cap \phi(I)$ and $\abs{N_G(\phi(u))\cap \phi(I)}=1$.
    \end{itemize}
\end{corollary}
\begin{proof}
    Follows from Lemma~\ref{no2separateNeighbors}.
\end{proof}

A tree that is not isomorphic to a star has at least two leaves with different sole neighbors. 
By Corollary \ref{onlyshared}, if a component $\mathcal{C}$ of $H$ is isomorphic to such a tree, then it falls into one of the following three types:
\begin{itemize}
    \item[Type T1:] At least two vertices of $\mathcal{C}$ are mapped to vertices of $G$ with no neighbors in  $\phi(I)$.
    
    \item[Type T2:] One vertex in $\mathcal{C}$ is mapped to a vertex of $G$ with no neighbors in $\phi(I)$, and two other vertices are mapped to vertices $w$ and $w'$ of $G$ with $N_G(w)\cap \phi(I)=N_G(w') \cap \phi(I)$ and $\abs{N_G(w)\cap \phi(I)}=1$.
    
    \item[Type T3:] There are two disjoint pairs of vertices in $\mathcal{C}$ where each pair is mapped to a pair of vertices
    $w_1$,$w'_1$ and  $w_2$, $w'_2$ of $G$ with $N_G(w_i)\cap \phi(I)=N_G(w'_i) \cap \phi(I)$ and $\abs{N_G(w_i)\cap \phi(I)}=1$ for $i\in\{1,2\}$.
\end{itemize}

We enumerate these trees in the following way. 
Let $A, B$, and $C$ be the sets of components of Types $T1, T2$, and $T3$, respectively. The numbers of vertices of the trees in $A$ are denoted $a_1,a_2,\dots,a_{|A|}$, the numbers of vertices of the trees in $B$ are denoted $b_1,b_2,\dots,b_{|B|}$, and the numbers of vertices of the trees in $C$ are denoted $c_1,c_2,\dots,c_{|C|}$.

For components $\mathcal{C}$ isomorphic to stars on at least four vertices, there are four possibilities: 
\begin{itemize}
    \item [Type S1:] There is a vertex in $\phi(I)$ incident to the image of the center of the star and to the image of at least one leaf of the star. (Thus by Corollary~\ref{onlyshared}, this is the only vertex of $\phi(I)$ incident to a vertex in the image of the star $\mathcal{C}$.)
    
    \item[Type S2:]     Only the image of the center of the star $\mathcal{C}$ has neighbors in  $\phi(I)$.
    
    \item[Type S3:] The image of at least one leaf of $\mathcal{C}$ has a neighbor in $\phi(I)$, but the image of the center does not have any.
    \item[Type S4:] No vertex from $\phi(V(\mathcal{C}))$ has a neighbor in  $\phi(I)$.
\end{itemize}

Let $D, E, F$, and $L$ be the sets of components of Types $S1, S2, S3$, and $S4$, respectively. The numbers of vertices of the stars in $D$ are denoted $d_1,d_2,\dots,d_{|D|}$, the numbers of vertices of the stars in $E$ are denoted $e_1,e_2,\dots,e_{|E|}$, the numbers of vertices of the stars in $F$ are denoted $f_1,f_2,\dots,f_{|F|}$, and the numbers of vertices of the stars in $L$ are denoted $l_1,l_2,\dots,l_{|L|}$.

For components $\mathcal{C}$ isomorphic to trees with  three vertices, there are two possibilities: 
\begin{itemize}
    \item [Type S5:] There is at most one vertex in $\phi(V(\mathcal{C}))$ adjacent to a vertex from $\phi(I)$.
    
    \item[Type S6:]     There is more than one vertex in $\phi(V(\mathcal{C}))$ adjacent to a vertex from $\phi(I)$.
    
\end{itemize}

 Let $m$ be the number of stars of Type S5 and $z$ be the number of stars of Type S6. 
 Lastly, let $n_0$ be the total number of vertices in the components of $H$ which contain a cycle. 

\begin{lemma}\label{lem:matching_cherry_odd_neighbouts_in_I}
    There are no two distinct vertices in $\phi(I)$ incident to the images of distinct leaves of 
    a matching edge of type M4 or to the images of distinct leaves of a star of type S6.
\end{lemma}
\begin{proof}
The statement trivially holds for M4 since it is a corollary of Lemma~\ref{no2separateNeighbors}. As for a star of type S6, the proof is similar to the proofs of Lemma~\ref{lem:noedge} and Lemma~\ref{no2separateNeighbors}.  
\end{proof}

The following table briefly summarizes the structure of the skeleton of $H$.

\begin{center}
\begin{tabular}{ c | c | c  }
Type & $\#$ of components & $\#$ of vertices\\
 \hline
 Matching Type M1 & $r_1$ & $2r_1$\\ 
 Matching Type M2 & $r_2$ & $2r_2$\\  
 Matching Type M3 & $r_3$ & $2r_3$\\ 
 Matching Type M4 & $y$ & $2y$\\ 
 Tree Type T1 &  $\abs{A}$ & $a_1+a_2+\cdots+a_{\abs{A}}$\\  
 Tree Type T2 &  $\abs{B}$ & $b_1+b_2+\cdots+b_{\abs{B}}$\\ 
 Tree Type T3 &  $\abs{C}$ & $c_1+c_2+\cdots+c_{\abs{C}}$\\  
 Star Type S1 &  $\abs{D}$ & $d_1+d_2+\cdots+d_{\abs{D}}$\\ 
 Star Type S2 &  $\abs{E}$ & $e_1+e_2+\cdots+e_{\abs{E}}$\\  
 Star Type S3 &  $\abs{F}$ & $f_1+f_2+\cdots+f_{\abs{F}}$\\ 
 Star Type S4 &  $\abs{L}$ & $l_1+l_2+\cdots+l_{\abs{L}}$\\  
 Star Type S5 &  m & $3m$\\ 
 Star Type S6 &  z  & $3z$\\
Non-tree components& & $n_0$  
\end{tabular}\label{table}
\end{center}

\subsubsection{Establishing bounds on the number of edges of  \texorpdfstring{$H$}{Lg} and  \texorpdfstring{$G$}{Lg}}\label{sub:counts}

After characterizing all of the connected components of $H$, we are ready to bound the number of edges of $H$.

\begin{align}\label{eq:lowerbound_skeleton_edges}
\begin{split}
e(H)&\ge n_0+\left(\sum_{i=1}^{\abs{A}} (a_i-1)+\sum_{i=1}^{\abs{B}}(b_i-1)+\sum_{i=1}^{\abs{C}}(c_i-1)\right)
\\ &+\left(\sum_{i=1}^{\abs{D}}(d_i-1)+\sum_{i=1}^{\abs{E}}(e_i-1)+\sum_{i=1}^{\abs{F}}(f_i-1)+\sum_{i=1}^{\abs{L}}(l_i-1)\right)
+2\left(m+z\right)+\left(r_1+r_2+r_3+y\right).
\end{split}
\end{align}

In the upcoming analysis, we will sum over all trees of a specific type, omitting explicit indices in summations for brevity.

To determine $e(G)$, we must keep track of edges between two vertices that are each the image of skeleton vertices as well as edges between a vertex in the image of $I$ and a vertex in the image of the skeleton.  By Corollary~\ref{Cor:independetnt}, the image of $I$ is an independent set.

The number of edges between two vertices that are each the image of skeleton vertices is the number of edges in the induced subgraph on the images of the skeleton vertices. This is a planar graph, so the number of edges is less than three times the number of vertices. 
Note however that we can improve this bound when $r_1$ or $r_2$ is positive since there are at least $r_1$ vertices of degree $1$, one endpoint of the image of each Type M1 matching edge,  and at least $r_2$ others of degree $2$, one endpoint of the image of each Type M2 matching edge. Furthermore, by Lemma~\ref{lem:noedge}, these low-degree vertices induce an independent set.

Thus, the number of edges in $G$ in the induced subgraph on the image of the skeleton is less than
\begin{align}\label{eq:skeleton_edges}
\begin{split}
3&\left(n_0+\sum a_i+\sum b_i+\sum c_i+\sum d_i+\sum e_i+\sum f_i+\sum l_i\right)\\ +  3&\left(2(r_1+r_2+r_3+y)+3(m+z)\right)-2r_1-r_2.    
\end{split}
\end{align}

Define the set $P$ consisting of the following vertices:
\begin{itemize}
    \item For each tree of Type T2, include the images of two vertices: namely the images of a leaf and its sole neighbor which have a sole common neighbor in $\phi(I)$. 
    \item For each tree of Type T3, include the images of four vertices: namely the images of two pairs of a leaf and its sole neighbor which have a sole common neighbor in $\phi(I)$. 
    
    \item For each star of Type S1, include the images of all vertices. 
    \item For each matching edge of Type M4, include the images of both vertices.
    \item For each star of Type S6, include the images of all vertices.
\end{itemize}

The number of edges between a vertex in $P$ and a vertex in $\phi(I)$ is at most 
\[
2|B|+4|C|+\sum d_i + 2y +3z.
\]

Removing these edges changes the neighborhoods of some vertices in $\phi(I)$ and potentially causes there to now be two degree $1$ or degree $2$ vertices in $\phi(I)$ with the same neighborhood. To avoid this issue, if some vertex in $\phi(I)$ has at least one neighbor in $P$, but at most two neighbors not in $P$, we will also remove its remaining one or two incident edges.

Note that the number of vertices in $\phi(I)$ with at least one neighbor in $P$ is at most 
\[
|B|+2|C|+|D|+y+z.
\]

Thus the number of edges incident to vertices of $\phi(I)$ with at least one neighbor in $P$ is at most $(2|B|+4|C|+\sum d_i + 2y +3z)+2(|B|+2|C|+|D|+y+z)$, for a total upper bound of:
\begin{equation}\label{eq:otherI}
4|B|+8|C|+\sum (d_i+2)+4y+5z.
\end{equation}
Furthermore, in the graph obtained from $G$ by removing all these vertices and edges, the degree $1$ and degree $2$ vertices from $\phi(I)$ have distinct neighborhoods. 

Among the remaining vertices in $\phi(I)$, we let $J_i$ denote the set of such vertices with degree $i$. Because we have previously determined that there are no edges between two vertices in the image of $I$, we get that there are $i|J_i|$ edges incident to $J_i$ for each $i$ and these sets of edges are disjoint. 

Vertices in $J_1$ can be adjacent to vertices in the images of the components of the skeleton that have cycles, that are isomorphic to non-star trees, or that are stars of Type S2, S3, or S5. In particular, no vertex in $J_1$ has a neighbor in the image of a matching edge of $H$. Since no degree $1$ vertices in $\phi(I)$ have the same neighborhood, no vertex in the image of the skeleton has multiple neighbors in $J_1$. 
Thus, the size of $|J_1|$ is at most the number of vertices among skeleton vertices whose images might have
a degree 1 neighbor in the image of I. 
Recall at least two vertices of Type $T_1$ trees are mapped to vertices of $G$ with no neighbors in $\phi(I)$. 
Similarly, for Type $T_2$ and Type $T_3$ trees, there are $3$ and $4$ vertices, respectively whose images either have no neighbors in $\phi(I)$ or only neighbors accounted for in \ref{eq:otherI}.
Therefore, the number of edges in $G$ incident to a vertex from $J_1$ is at most
\begin{align}\label{eq:J_1}
n_0+\sum (a_i-2)+\sum (b_i-3)+\sum(c_i-4)+|E|+\sum(f_i-1)+m.
\end{align}

The vertices in $J_2$ can be adjacent to vertices in the images of the components of the skeleton that have cycles, that are isomorphic to non-star trees, or that are stars of Type S2, S3, or S5. Additionally, they can be adjacent to vertices in the images of matching edges of Types M1, M2, though only to one vertex per such edge by Corollary \ref{onlyshared}. 
Thus, there are at most 
\[
n_0+\sum (a_i-2)+\sum (b_i-3)+\sum(c_i-4)+|E|+\sum(f_i-1)+r_1+r_2+m
\]
vertices in the image of the skeleton which have a neighbor in $J_2$. 
We now construct an auxiliary graph on this vertex set where for every vertex in $J_2$ with neighborhood $\{u,v\}$, we add an edge between $u$ and $v$ in the auxiliary graph. 
Since no two degree $2$ vertices in $\phi(I)$ have the same neighborhood, this is a simple graph. As $G$ is planar, the auxiliary graph is also planar. 
Thus, the number of edges in the auxiliary graph, which is equal to $|J_2|$, is at most three times the number of vertices in the auxiliary graph. Thus we have an upper bound for $|J_2|$.  
The number of edges in $G$ incident to a vertex from $J_2$ is  $2|J_2|$ and is thus at most
\begin{align}\label{eq:J_2}
 6\left(n_0+\sum (a_i-2)+\sum (b_i-3)+\sum(c_i-4)+|E|+\sum(f_i-1)+r_1+r_2+m\right).
\end{align}

Lastly, we need to bound $\sum_{i=3}^{q} i|J_i|$, where $q$ is the maximum degree of a vertex in $\phi(I)$ with no neighbors in $P$. To do this, consider the subgraph, $S$, of $G$ whose edges are the edges incident to $\bigcup_{i=3}^q J_i$ along with the images of the edges of the components of $H$ that are stars of Type S3, 
and the vertices are simply any vertex of $G$ incident to at least one of these edges. 
We note that  $\phi(I)$ is an independent set in $G$ so $\bigcup_{i=3}^q J_i$ is independent in $S$ also. 
The images of the centers of Type S3 stars also have no neighbors in $\phi(I)$. Therefore, the union of $\bigcup_{i=3}^q J_i$ with the set of images of these star centers is an independent set in $S$, which we denote by $T$. Any edge in $S$ is between a vertex of $T$ and a vertex of $V(S)-{T}$, so $S$ is a bipartite planar graph; thus we have $e(S)\le 2v(S)$.
The vertices in the image of the skeleton that cannot have neighbors in $\bigcup_{i=3}^q J_i$ and thus are not vertices of $S$ are the images of at least two vertices per tree of Type T1, at least three per tree of Type T2, at least four per tree of Type T3, any vertex in stars of Type S1, S4, and S6, all the leaves of stars of Type S2, any vertex in a matching edge of Type M4, one vertex per matching edge of Types M1, M2, M3, and at least two vertices per star of Type S5.
\vspace{-5mm}

{\footnotesize 
\begin{align*}
    \sum_{i=3}^{q} i|J_i|+\sum (f_i-1)&\le
    2\left(n_0+\sum (a_i-2)+\sum (b_i-3)+\sum (c_i-4)+|E|+\sum f_i+\sum_{i=1}^3r_i+m+\sum_{i=3}^q |J_i|\right)\\
    \sum_{i=3}^{q} (i-2)|J_i|&\le
    2\left(n_0+\sum (a_i-2)+\sum (b_i-3)+\sum (c_i-4)+|E|+r_1+r_2+r_3+m\right)+\sum (f_i+1)\\
    \sum_{i=3}^q (3i-6)|J_i|&\le 6\left(n_0+\sum (a_i-2)+\sum (b_i-3)+\sum (c_i-4)+|E|+r_1+r_2+r_3+m\right)+3\sum (f_i+1).
\end{align*}}

As $i\le 3i-6$ for $i\ge 3$, we have that $\sum_{i=3}^q i|J_i|$ is bounded above by
\begin{equation}\label{eq:J_3}
    6\left(n_0+\sum (a_i-2)+\sum (b_i-3)+\sum (c_i-4)+|E|+r_1+r_2+r_3+m\right)+3\sum (f_i+1).
\end{equation}

Adding up \ref{eq:skeleton_edges}, \ref{eq:otherI}, \ref{eq:J_1}, \ref{eq:J_2}, and \ref{eq:J_3}, we get an upper bound on $e(G)$, 
\begin{align*}
    e(G)&< 16n_0+\sum(16a_i-26)+\sum(16b_i-35)+\sum(16c_i-44)+\sum(4d_i+2)+\sum(3e_i+13)\\&+\sum(13f_i-4)+\sum(3l_i)+16r_1+17r_2+12r_3+10y+22m+14z.
\end{align*}

Recall Equation~\ref{eq:lowerbound_skeleton_edges}, a lower bound for $e(H)$, 
\begin{align*} 
\begin{split}
e(H)&\ge n_0+\left(\sum (a_i-1)+\sum(b_i-1)+\sum(c_i-1)\right)\\
&+\left(\sum(d_i-1)+\sum(e_i-1)+\sum(f_i-1)+\sum(l_i-1)\right)+\left(r_1+r_2+r_3+y\right)+2\left(m+z\right).
\end{split}
\end{align*}

To check whether $e(H)>e(G)/16$, we use the following lemma.
\begin{lemma}\label{lem:compare}
For a real number $q$ and positive reals $x_i, y_i$ for $1\le i\leq k$, if $\frac{x_i}{y_i}\ge q$ for all $i$, then
\[
\frac{\sum_{i=1}^{k}x_i}{\sum_{i=1}^{k}y_i}\ge q.
\]
\end{lemma}
\begin{proof}
    The proof is straightforward.
\end{proof}
\subsubsection{Proof when  \texorpdfstring{$4r_3\geq r_2$}{Lg}}\label{sub:proof1}
To establish $\frac{e(H)}{e(G)}>1/16$, we can show that the lower bound for $e(H)$ is at least $1/16$ of the strict upper bound for $e(G)$. Using Lemma~\ref{lem:compare}, it suffices to show that every nonzero term in the bound for $e(H)$ is at least $1/16$ of the corresponding term for $e(G)$.
Note that all $a_i,b_i,c_i,d_i,e_i,f_i$, and $l_i$ are at least $4$, so we have 
\[\frac{a_i-1}{16a_i-26}, \frac{d_i-1}{4d_i+2}, \frac{e_i-1}{3e_i+13} \mbox{,~and~}\frac{f_i-1}{13f_i-4} \geq \frac{1}{16},\] 
and similar inequalities hold for the terms involving each $b_i,c_i,l_i$.

Clearly $\frac{n_0}{16n_0}$ and $\frac{r_1}{16r_1}$ are equal to $1/16$, and $\frac{2m}{22m},\frac{2z}{14z}, \frac{r_3}{12r_3}$, and $\frac{y}{10y}>1/16$. The only exception is $\frac{r_2}{17r_2}=1/17$, so all we have shown so far is that $\frac{e(H)}{e(G)}>1/17$. Making use of the fact that $\frac{r_3}{12r_3}$ is strictly greater than $1/16$ for nonzero $r_3$, we can choose to group the $r_2$ and $r_3$ terms together while keeping everything else separate. 
By Lemma~\ref{lem:compare}, $\frac{e(H)}{e(G)}>1/16$ holds when
\begin{align*}
    \frac{r_2+r_3}{17r_2+12r_3}&\ge\frac{1}{16}\\
    4r_3&\ge r_2.
\end{align*}

Thus, we have established the desired bound whenever $r_3$ is sufficiently large relative to $r_2$. Next, we will bound $e(G)$ in a different manner that gives the desired result $\frac{e(H)}{e(G)}>1/16$ provided that $r_3$ is sufficiently \textit{small} relative to $r_2$. All possible combinations of $r_2$ and $r_3$ will be covered by these cases, thus completing the proof.

\subsubsection{Proof when  \texorpdfstring{$4r_3\leq r_2$}{Lg}}\label{sub:proof2}
We use the same bounds as before on the number of edges incident to vertices in  $\phi(I)$, except for those in $J_2$. 
We also will make a new estimate for the number of edges in the induced graph on the image of the skeleton. 
In particular, we find a bound on this total number of edges as part of one combined process.

Recall that for each matching edge of $H$ of Type M2, there is a vertex that is embedded to a degree $2$ vertex of $G$. Let us denote the set of all these degree $2$ vertices by $R$. 
Note that the other vertex of the matching edge embeds to a vertex with multiple neighbors in $\phi(I)$. Thus, the vertices from $R$ have no neighbors in the image of $I$, by Corollary \ref{onlyshared}. 
Furthermore, $G[R]$ is an empty graph by Lemma~\ref{lem:noedge}, so $G[J_2\cup R]$ is an empty graph also.
 Note that since no vertex of $R$ has neighbors in $\phi(I)$, the removal of edges incident to the vertices in  $\phi(I)$ that had neighbors in $P$, which we performed before defining $J_1, J_2, \dots, J_q$, did not change the neighborhood of any vertex in $R$. Thus because $G$ initially had no degree $2$ twins, all vertices in $J_2\cup R$ have distinct neighborhoods from each other. Furthermore, for each such vertex, its two neighbors lie in the image of the skeleton. 

We wish to find an upper bound on the number of edges between two vertices in the image of the skeleton where neither is in $R$. 
We also wish to find an upper bound on the number of edges incident to $J_2\cup R$. 
For this first quantity, it is the number of edges in an induced subgraph of $G$ with $n_0+\sum a_i+\sum b_i+\sum c_i+\sum d_i+\sum e_i+\sum f_i+\sum l_i+2r_1+r_2+2r_3+2y+3m+3z$ vertices. Since $G$ is planar, this is at most $3(n_0+\sum a_i+\sum b_i+\sum c_i+\sum d_i+\sum e_i+\sum f_i+\sum l_i+2r_1+r_2+2r_3+2y+3m+3z)$ edges. 
Furthermore, since at least $r_1$ of these vertices have degree $1$ in $G$ and they are not adjacent to each other, 
we can improve the upper bound to $3(n_0+\sum a_i+\sum b_i+\sum c_i+\sum d_i+\sum e_i+\sum f_i+\sum l_i+2r_1+r_2+2r_3+2y+3m+3z)-2r_1$.

The number of edges incident to $J_2\cup R$ is $2|J_2\cup R|$. 
To bound $|J_2\cup R|$, consider an auxiliary graph $O_{J_2\cup R}$ where the vertex set is $N_G(J_2\cup R)$. 
For every vertex in $J_2\cup R$ with neighborhood $\{u,v\}$ in $G$, we add an edge between $u$ and $v$ in the graph $O_{J_2\cup R}$.
Since no two vertices in $J_2\cup R$ have the same neighborhood, this is a simple graph, and since $G$ is planar, the graph $O_{J_2\cup R}$ is also planar. Thus, the number of edges is at most $3|N_G(J_2\cup R)|$.
We now prove a structural result to establish an upper bound on $|N_G(J_2\cup R)|$.

\begin{claim}\label{RlikeJ2}
    If $\mathcal{C}$ is a component of the skeleton of $H$ that is isomorphic to a tree, $u$ is a leaf of $\mathcal{C}$, and $v$ is the sole neighbor of $u$ in $\mathcal{C}$, the following scenarios are all impossible:
    \begin{enumerate}[label={(\roman*)},itemindent=1em]
        \item The vertices $\phi(u)$ and $\phi(v)$ both have a neighbor in $R$. \label{itm:pt1}
        \item The vertex $\phi(u)$ has a neighbor in $R$ and the vertex $\phi(v)$ has a neighbor in $J_2$. \label{itm:pt2}
        \item The vertex $\phi(u)$ has a neighbor in $J_2$ and the vertex $\phi(v)$ has a neighbor in $R$.\label{itm:pt3}
    \end{enumerate}
\end{claim}   

\begin{proof}
A vertex in $R$ is incident to two edges in $G$, one of which is the image of the matching edge. 
Hence if a pair of vertices in $G$, that are not images of the vertices in the matching edges of Type M2, each have a neighbor in $R$, these are distinct neighbors.

 If $\mathcal{C}$ is a matching edge of Type M2, all parts hold automatically from the fact that $J_2\cup R$ is an independent set. 
 
Proof of Claim~\ref{RlikeJ2}~\ref{itm:pt1}: For this scenario to hold, the images of $u$ and $v$ have distinct neighbors in $R$, by the first paragraph of this proof. 
Suppose that the preimages of these distinct neighbors are $w_1$ and $w_2$, respectively. 
The vertices $w_1$ and $w_2$ are vertices of matching edges in $H$ and we denote the other vertices of the matching edges by $x_1$ and $x_2$ respectively. 
Each of the vertices $\phi(x_1)$ and $\phi(x_2)$ has at least two neighbors in $\phi(I)$ so we can choose distinct vertices $q_1, q_2\in I$ such that $\phi(q_i)\phi(x_i)$ is an edge of $G$ for $i=1,2$. 
By Observation~\ref{obs:all_I_in_one_face}, $q_1$ and $q_2$ are embedded in the same face of $H$. 

Let us consider the following embedding, derived from $\phi$ by the following modification: 
\begin{itemize}
    \item Map $u$ to $\phi(w_2)$;
    \item Map $w_2$ to  $\phi(q_2)$;
        \item Map $q_2$ to $\phi(x_1)$;
    \item Map $x_1$ to $\phi(u)$;
\end{itemize}
This is a valid embedding of the tree and of the matching edges $w_1x_1$ and $w_2x_2$, but now $q_1$ and $q_2$ are mapped to adjacent vertices of $G$. 
By Observation~\ref{obs:all_I_in_one_face}, it is possible to add an edge between them in $H$ without introducing a crossing. The resulting graph is still a subgraph of $G$,  contradicting the assumption that $H$ is a plane-saturated subgraph.

Proof of Claim~\ref{RlikeJ2}~\ref{itm:pt2}: Suppose that the vertex $\phi(u)$ has a neighbor in $R$ with preimage $w_1$ and that the vertex $\phi(v)$ has a neighbor in $J_2$ with preimage $w_2$. 
$w_1$ is a vertex of a matching edge in $H$ and we let $x_1$ be the other vertex of the matching edge. 
The vertex $\phi(x_1)$ has a neighbor $\phi(q)\in \phi(I)$ which is distinct from $\phi(w_2)$ since it has at least two neighbors in $\phi(I)$. 

Let us consider the following embedding, generated from $\phi$: 
\begin{itemize}
    \item Map $u$ to $\phi(w_2)$;
        \item Map $w_2$ to $\phi(x_1)$;
    \item Map $x_1$ to  $\phi(u)$;

\end{itemize}
This is a valid embedding of the tree and the matching edge $w_1x_1$, but now $q,w_2\in I$ are mapped to adjacent vertices of $G$.
By Observation~\ref{obs:all_I_in_one_face}, it is possible to add an edge between them in $H$ without introducing a crossing. The resulting graph is still a subgraph of $G$, contradicting the assumption that $H$ is a plane-saturated subgraph.

Proof of Claim~\ref{RlikeJ2}~\ref{itm:pt3}: 
Suppose that the vertex $\phi(u)$ has a neighbor in $J_2$ with preimage $w_1$ and that the vertex $\phi(v)$ has a neighbor in $R$ with preimage $w_2$. 
The vertex $w_2$ is a vertex of a matching edge in $H$ and we let $x_2$ be the other vertex of the matching edge.
The vertex $\phi(x_2)$ has a neighbor $\phi(q)\in \phi(I)$ that is distinct from $\phi(w_1)$, since it has at least two neighbors in $\phi(I)$. 
Let us consider the following embedding, generated from $\phi$: 
\begin{itemize}
    \item Map $u$ to $\phi(w_2)$;
        \item Map $w_2$ to $\phi(q)$;
    \item Map $q$ to  $\phi(u)$;

\end{itemize}
This is a valid embedding of the tree and the matching edge $w_2x_2$, but now $q,w_1\in I$ are mapped to adjacent vertices of $G$. 
By Observation~\ref{obs:all_I_in_one_face}, it is possible to add an edge between them in $H$, contradicting the assumption that $H$ is a plane-saturated subgraph.
\end{proof}


Any component of $H$ isomorphic to a non-star tree has two disjoint pairs of adjacent vertices where the image of at least one has no neighbor in $J_2$ by Corollary \ref{onlyshared}. 
For such a pair of vertices, if the image of exactly one has a neighbor in $J_2$, then by Claim \ref{RlikeJ2} \ref{itm:pt2} and \ref{itm:pt3}, the image of the other has no neighbor in $R$. 
If the image of neither has a neighbor in $J_2$, then by Claim \ref{RlikeJ2} \ref{itm:pt1}, the image of at least one has no neighbor in $R$. 
In all scenarios, the image of at least one vertex from the pair has no neighbors in $J_2\cup R$.
Note that any component of $H$ isomorphic to a star (including the matching edges of Types M1, M3, M4) has one such pair of vertices. 

Thus, the number of vertices of $O_{J_2\cup R}$ is at most
\begin{align}\label{eq:JR}
\begin{split}
n_0+\sum(a_i-2)+\sum(b_i-2)+\sum(c_i-2)&+\sum(d_i-1)+\sum(e_i-1)+\sum(f_i-1)\\
&+\sum (l_i-1)
+r_1+r_2+r_3+y+2m+2z.
\end{split}
\end{align}

Since $O_{J_2\cup R}$ is planar, the number of edges is at most three times the value in \eqref{eq:JR}.
The number of edges of $G$ incident to some vertex of $J_2\cup R$ is at most six times the value in \eqref{eq:JR}.


Recall the upper bounds for the number of remaining edges in $G$, which is the same as in the previous strategy:
\begin{align*}
\begin{split}
&\left(4|B|+8|C|+\sum (d_i+2)+4y+5z\right)\\
&+\left(n_0+\sum (a_i-2)+\sum (b_i-3)+\sum(c_i-4)+|E|+\sum(f_i-1)+m\right)\\&+\left(6\left(n_0+\sum (a_i-2)+\sum (b_i-3)+\sum (c_i-4)+|E|+r_1+r_2+r_3+m\right)+3\sum (f_i+1)\right).
\end{split}
\end{align*}

Adding everything up, we get
\begin{align*}
\begin{split}
    e(G)&<16n_0+\sum(16a_i-26)+\sum(16b_i-29)+\sum(16c_i-32)+\sum(10d_i-4)+\sum(9e_i+1)\\&+\sum(13f_i-4)+\sum(9l_i-6)+16r_1+15r_2+18r_3+16y+28m+26z.
\end{split}
\end{align*}

As before, we will apply Lemma~\ref{lem:compare}. We will compare each term in our lower bound for $e(H)$ to the corresponding term in our upper bound for $e(G)$, with the one exception being that the $r_2$ and $r_3$ terms will be considered together. 
Some of the terms are the same as before, so besides the $r_2$ and $r_3$ term, it suffices to check that $\frac{e_i-1}{9e_i+1}\ge 1/16$ for $e_i\ge 4$, and that $ \frac{2m}{28m}\ge 1/16$, which are both true. The only term in the lower bound for $e(H)$ that is not at least $1/16$ times the corresponding term in the upper bound for $e(G)$ is the $r_3$ term since this would give $\frac{r_3}{18r_3}=1/18$. However, we are grouping the $r_2$ and $r_3$ terms together and it suffices to have
\begin{align*}
    \frac{r_2+r_3}{15r_2+18r_3}&\ge 1/16\\
    r_2&\ge 2r_3.
\end{align*}

We previously established that the desired bound holds when $4r_3\ge r_2$, so this bound of $r_2\ge 2r_3$ holds in all cases where we had not already proven the desired claim. Thus, the proof that $\frac{e(H)}{e(G)}>1/16$ is complete.

\section{Proof of Theorem~\ref{thm:mostK}}\label{section:mostK}
\subsection{Modifying the upper bound construction}

\begin{construction}\label{generalconstruction}
Let $m\geq{9}$ and $G_0$ be a graph on $m+2$ vertices consisting of a $C_m$ and two additional vertices which are each adjacent to every vertex of the $C_m$.
There exists a unique planar embedding of $G_0$. 
Note that $G_0$ has $m+2$ vertices, $3m$ edges and $2m$ faces.

Let $G_1$ be the planar graph obtained from $G_0$ in the following way. 
For each face $f$, we add a vertex $u_f$, of degree $3$ adjacent to the three boundary vertices of $f$. 
For each edge $e$ of $G_0$, we add $k_2$ vertices of degree $2$ which are each adjacent to the two endpoints of $e$. 
For each vertex $v$ of $G_0$, we add $k_1$ vertices of degree $1$ which are each adjacent to $v$. 

Let $G$ be the disjoint union of $G_1$ and  $K_4$.
\end{construction}

Note that $G$ is planar, at most $k_1$ of its degree $1$ vertices have the same neighborhood, and at most $k_2$ of its degree $2$ vertices have the same neighborhood.

\begin{claim}
    Let $G$ be as it is defined in Construction \ref{generalconstruction}. Then,
    $\psr(G)<\frac{1}{9+k_1+6k_2}+\frac{3}{m}$.
\end{claim}

The number of vertices of $G$ is
\[
v(G)=(m+2)+2m+k_2(3m)+k_1(m+2)+4=(3k_2+k_1+3)m+(2k_1+6).
\]

The number of edges of $G$ is
\[
e(G)=3m+3(2m)+2k_2(3m)+k_1(m+2)+6=(9+k_1+6k_2)m+(2k_1+6).
\]

Consider a planar graph $H_0$ consisting of the disjoint union of $C_m$, $K_4$, two copies of $K_{1,k_1+4k_2+9}$, and $(2+k_1+3k_2)m-(8k_2+18)$ isolated vertices. The number of vertices of $H_0$ is $v(H_0)=(3k_2+k_1+3)m+(2k_1+6)=v(G)$ and the number of edges is $e(H_0)=m+6+2(k_1+4k_2+9)=m+(2k_1+8k_2+24)$.

It is easy to see that $H_0$ is a subgraph of $G$. Now let $H$ denote the planar embedding of $H_0$ as follows: We first draw the $K_4$ without crossings so it has four triangular faces. Inside one of the faces, we embed $C_m$, inside another we embed two $K_{1,k_1+4k_2+9}$'s, and inside a third, we embed $(2+k_1+3k_2)m-(8k_2+18)$ isolated vertices.

To demonstrate that $H$ is plane-saturated with respect to $G$, we observe that every vertex of $G$ has a degree at most $3$ except the vertices in $V(G_0)$: the $C_m$ and the two high-degree vertices adjacent to all vertices of that $C_m$. A vertex from $V(C_m)$ has degree $k_1+4k_2+8$ in $G$. Indeed, it is adjacent to two other vertices on the $C_m$, the two high-degree vertices, four additional vertices of degree $3$, $4k_2$ additional vertices of degree $2$, and $k_1$ additional vertices of degree $1$. Each of the two high-degree vertices in $G_0$ has degree $(2+k_2)m+k_1\ge k_1+4k_2+9$. As $G$ has only two vertices of degree at least $k_1+4k_2+9$, the two degree $k_1+4k_2+9$ vertices in $H$ can be uniquely identified (up to swapping) in the original graph $G$. The proof that $H$ is plane-saturated for $G$ follows similarly to that of Claim \ref{ubworks}.
We then observe that
\[
\frac{e(H)}{e(G)}=\frac{m+(2k_1+8k_2+24)}{(9+k_1+6k_2)m+(2k_1+6)}<\frac{1}{9+k_1+6k_2}+\frac{3}{m},
\]

and thus approaches $\frac{1}{9+k_1+6k_2}$ as $m\to\infty$.

\subsection{Modifying the lower bound proof}

For a planar $G$ and a plane-saturated subgraph $H$ where each face contains at most one isolated vertex, we note that as per the proof of Lemma~\ref{facewith2}, that $\frac{e(H)}{e(G)}\ge 1/6>\frac{1}{9+k_1+6k_2}$. It follows that we may assume every embedding of $H$ in $G$ is such that the image of $I$ is an independent set.

We now modify the two-phase algorithm from Section \ref{section:twinfree}. In the new Phase~\ref{Phase_one},

\begin{enumerate}
    \item[(a)] Select a matching edge $e$ of the skeleton such that $\phi(e)=ab$, $|N(b)\cap \phi(I)|\ge 2$, there exists a vertex $c\in N(b)\cap \phi(I)$ with $d_G(c)=1$, and $d_G(a)>1$.
    \item[(b)] We may instead embed this matching edge to the edge $bc$ of $G$, placing $a$ in the image of $I$, and not changing anything else. Thus we redefine $\phi(\phi^{-1}(a)):=c$ and $\phi(\phi^{-1}(c)):=a$.
    \item[(c)] We repeat this process until there are no more matching edges of this form. The process is guaranteed to terminate because on each step, the number of vertices in $\phi(I)$ of degree $1$ in $G$ is strictly decreasing.
\end{enumerate}

Phase~\ref{phase_two} will proceed as in Section \ref{section:twinfree}.
The structural properties proved in Section \ref{section:twinfree} still apply and the various types of skeleton components are exactly as in Table \ref{table}, except that we modify the definition slightly for matching edges of Type M1. In this section, a component of $H$ consisting of just the edge $e$ is a Type M1 matching edge if one vertex of $\phi(e)$ has degree $1$ in $G$ and the other has at least two neighbors in $\phi(I)$. Namely, we drop the condition on the degrees of those neighbors. Now, we must simply modify the resulting bounds for $e(G)$.
In particular, changes must be made to the bounds on $|J_1|$ and $|J_2|$ in Subsection \ref{sub:counts}.

Each of the images of the vertices in the components of the skeleton that have cycles, that are isomorphic to non-star trees, or that are stars of Type S2, S3, or S5 can have up to $k_1$ neighbors in $J_1$. Recall at least two vertices of Type $T_1$ trees are mapped to vertices of $G$ with no neighbors in $\phi(I)$. 
Similarly, for Type $T_2$ and Type $T_3$ trees, there are $3$ and $4$ vertices, respectively whose images either have no neighbors in $\phi(I)$ or only neighbors accounted for in \ref{eq:otherI}. Furthermore, the image of a vertex in a matching edge of Type M1 that has neighbors in $\phi(I)$ may have up to $k_1-1$ neighbors in $J_1$ since it has at most $k_1$ degree $1$ neighbors, including the image of the other vertex of this matching edge. 
Thus,
\begin{align}\label{eq:newJ_1}
\begin{split}
|J_1|&\le k_1\big(n_0+\sum (a_i-2)+\sum (b_i-3)+\sum(c_i-4)+|E|\\
&+\sum(f_i-1)+m\big)+(k_1-1)r_1.   
\end{split}
\end{align}

The skeleton vertices whose images can be neighbors of the vertices in $J_2$ are precisely those described in Section \ref{section:twinfree}. However, as it is now possible for $k_2$ degree $2$ vertices to have the same neighborhood in $G$, the auxiliary graph we construct is now a multigraph. The vertices of this multigraph are the vertices adjacent to some element of $J_2$ and for every vertex of $J_2$, we add an edge in the multigraph between its neighbors in $G$. This multigraph has at most $k_2$ times as many edges as a planar graph on the same vertex set, so we get:
\begin{align}\label{eq:newJ_2}
2|J_2|\le 6k_2\left(n_0+\sum (a_i-2)+\sum (b_i-3)+\sum(c_i-4)+|E|+\sum(f_i-1)+r_1+r_2+m\right).
\end{align}

Replacing the contributions of \ref{eq:J_1} and \ref{eq:J_2} with those of \ref{eq:newJ_1} and \ref{eq:newJ_2} in our upper bound for $e(G)$ yields
\begin{align*}
\begin{split}
  e(G)&< (k_1+6k_2+9)n_0+\sum\big((9+k_1+6k_2)a_i-(2k_1+12k_2+12)\big)\\ 
  &+\sum\big((9+k_1+6k_2)b_i-(3k_1+18k_2+14)\big)+\sum\big((9+k_1+6k_2)c_i-(4k_1+24k_2+16)\big)\\
  &+\sum(4d_i+2)+\sum\big(3e_i+(k_1+6k_2+6)\big)+\sum\big((k_1+6k_2+6)f_i-(k_1+6k_2-3)\big)\\
  &+\sum(3l_i)+(k_1+6k_2+9)r_1+(6k_2+11)r_2+12r_3+10y+(k_1+6k_2+15)m+14z.
\end{split}
\end{align*}

Recall that \begin{align*} 
\begin{split}
e(H)&\ge n_0+\left(\sum (a_i-1)+\sum(b_i-1)+\sum(c_i-1)\right)\\
&+\left(\sum(d_i-1)+\sum(e_i-1)+\sum(f_i-1)+\sum(l_i-1)\right)+\left(r_1+r_2+r_3+y\right)+2\left(m+z\right).
\end{split}
\end{align*}

We will use Lemma~\ref{lem:compare} to show that $\frac{e(H)}{e(G)}>\frac{1}{9+k_1+6k_2}$ for many pairs $k_1, k_2$. It suffices to check that the following expressions are all at least $\frac{1}{9+k_1+6k_2}$:
\begin{align*}
&\frac{a_i-1}{(9+k_1+6k_2)a_i-(2k_1+12k_2+12)}, \frac{d_i-1}{4d_i+2}, \frac{e_i-1}{3e_i+(k_1+6k_2+6)},\\& \frac{f_i-1}{(k_1+6k_2+6)f_i-(k_1+6k_2-3)},\frac{r_2}{(6k_2+11)r_2}, \frac{r_3}{12r_3},\frac{2m}{(k_1+6k_2+15)m}.
\end{align*}

As $k_1,k_2$ are non-negative and $a_i,d_i, e_i,f_i$ are always at least $4$, the first four ratios, as well as $\frac{2m}{(k_1+6k_2+15)m}$, are always at least $\frac{1}{9+k_1+6k_2}$. We additionally require:
$k_1\ge 2$ and $k_1+6k_2\ge 3$. Satisfying $k_1\ge 2$ automatically satisfies $k_1+6k_2\ge 3$, unless $k_1=2, k_2=0$. It remains to prove Theorem~\ref{thm:mostK} in the cases where $k_1=1$. 

The case $k_1=1,k_2=1$ was already handled in Section \ref{section:twinfree}, and our general proof proceeds in a similar manner by incorporating $k_1, k_2$ into the quantities considered in Subsection~\ref{sub:proof2}. First, we note that the $r_2$ and $r_3$ terms can be considered together so it suffices to have \[
\frac{r_2+r_3}{(6k_2+11)r_2+12r_3}\ge\frac{1}{9+k_1+6k_2}=\frac{1}{10+6k_2}.
\]
Rearranging, we see that the desired bound is established if $(6k_2-2)r_3\ge r_2$.

We then find an alternative bound for $e(G)$, following Subsection \ref{sub:proof2}. The only change is that since we are now allowed to have up to $k_2$ degree $2$ vertices with the same neighborhood, the upper bound on the number of edges incident to $J_2\cup R$ is now $6k_2$ times the value in \ref{eq:JR}. Incorporating this term, along with our new upper bound on $|J_1|$ gives
\begin{align*}
\begin{split}
    e(G)&<(9+k_1+6k_2)n_0+\sum\big((9+k_1+6k_2)a_i-(12+2k_1+12k_2)\big)+\sum\big((9+k_1+6k_2)b_i\\
    &-(14+3k_1+12k_2)\big)+\sum\big((9+k_1+6k_2)c_i-(16+4k_1+12k_2)\big)\\
    &+\sum\big((4+6k_2)d_i+(2-6k_2)\big)+\sum\big((3+6k_2)e_i+(6+k_1-6k_2)\big)\\
    &+\sum\big((6+k_1+6k_2)f_i+(3-k_1-6k_2)\big)+\sum\big((3+6k_2)l_i-6k_2\big)\\
    &+(9+k_1+6k_2)r_1+(6k_2+9)r_2+(6k_2+12)r_3+(6k_2+10)y\\
    &+(15+k_1+12k_2)m+(12k_2+14)z.
\end{split}
\end{align*}

Using Lemma~\ref{lem:compare}, it suffices to check that the following values are all at least $\frac{1}{9+k_1+6k_2}$:
\begin{align*}
    &\frac{a_i-1}{(9+k_1+6k_2)a_i-(12+2k_1+12k_2)}, \frac{d_i-1}{(4+6k_2)d_i+(2-6k_2)},\frac{e_i-1}{(3+6k_2)e_i+(6+k_1-6k_2)},\\
    &\frac{f_i-1}{(6+k_1+6k_2)f_i+(3-k_1-6k_2)},\frac{r_2+r_3}{(6k_2+9)r_2+(6k_2+12)r_3}, \frac{2m}{(15+k_1+12k_2)m}, \frac{2z}{(12k_2+14)z}.
\end{align*}

All of these desired inequalities follow from $k_1,k_2\ge 0$ and $a_i,d_i,e_i,f_i\ge 4$ except for 
\[
\frac{r_2+r_3}{(6k_2+9)r_2+(6k_2+12)r_3}\ge \frac{1}{9+k_1+6k_2}=\frac{1}{10+6k_2}.
\]

Rearranging, we see that this is true when $r_2\ge 2r_3$. As long as $k_2\ge 1$, one of the inequalities $(6k_2-2)r_3\ge r_2$ and $r_2\ge 2r_3$ holds, so we have proven Theorem~\ref{thm:mostK} in the cases where $k_1=1$.

\section{Conclusion}\label{section:conc}


We have already shown that $\psr(G)$ is bounded away from $0$ for graphs in the class $\mathcal{G}_{k_1,k_2}$ for all pairs of non-negative integers $k_1,k_2$. We conjecture that Theorem~\ref{thm:mostK} can be extended to all such pairs.

\begin{conjecture}\label{conj:allK}
    For non-negative integers $k_1, k_2$, any graph $G\in \mathcal{G}_{k_1,k_2}$ satisfies
     \[
    \psr(G)>\frac{1}{9+k_1+6k_2}.
    \]
\end{conjecture}

This is the best possible bound due to Construction \ref{generalconstruction}.
Of greatest interest is the case $k_1=k_2=0$, where Conjecture \ref{conj:allK} would give the following.

\begin{conjecture}
    For a planar graph $G$ with minimum degree at least $3$,
    \[
    \psr(G)>1/9.
    \]
\end{conjecture}



We can also impose further conditions on the plane-saturated subgraph. 

\begin{question}
    What is the smallest possible value of $\frac{e(H)}{e(G)}$ where $G$ is planar and $H$ is a plane-saturated subgraph of $G$ with no isolated vertices?
\end{question}

As $e(G)<3v(G)$ and the minimal degree condition on $H$ gives $e(H)\geq{\frac{v(H)}{2}}=\frac{v(G)}{2}$, we always have $\frac{e(H)}{e(G)}>\frac{1}{6}$. Our best-known construction in this setting comes from Example \ref{ex:onefifth} which gets arbitrarily close to $\frac{1}{5}$.

While this topic has not been studied before, we see potential for its further development. Besides the direction considered in our work, it may be interesting to compute $\psr(G)$ for various subclasses of planar graphs. It is also possible to extend our notion of plane-saturated to other genus surfaces or to study this problem for different classes of graphs, such as $k$-planar graphs.

Additionally, we wish to suggest some questions that motivated our study. 
This question arises from a two-player game played on a plane where a planar graph $G$ is fixed and $v(G)$ unlabeled vertices are embedded in the plane. Players take turns adding edges to the existing drawing while maintaining the drawing's planarity and at each point of the game, the graph drawn on the plane must be a subgraph of the fixed planar graph $G$. The game ends when a player is unable to make a move, i.e. when the drawing is a plane-saturated subgraph of $G$.
This setup leads to several questions of interest:
\begin{itemize}
    \item Suppose that Player 1 aims to minimize the number of edges in the final drawing $H$, while Player 2 aims to maximize the number of edges in $H$. Assuming optimal play for both players, we can define the \textit{game plane-saturation ratio} of $G$ as $\frac{e(H)}{e(G)}$ and ask which graphs $G$ minimize or maximize this quantity. 

    \item If the objective is instead to be the last player to add an edge, for which graphs $G$ does Player 1 have a winning strategy?
\end{itemize}

We hope these games will provide interesting challenges as well as recreation.

\section*{Acknowledgements}
The authors would like to thank Ben Lund for helpful conversations. AC was supported by the Institute for Basic Science IBS-R029-C1.
NS was supported by the Hungarian National Research, Development and Innovation Office NKFIH grant K132696.  
\bibliographystyle{alpha}
\bibliography{Proposal.bib}

\begin{bibdiv}
\begin{biblist}

\bib{bollobas1965generalized}{article}{
      author={Bollob{\'a}s, B{\'e}la},
       title={On generalized graphs},
        date={1965},
     journal={Acta Mathematica Hungarica},
      volume={16},
      number={3-4},
       pages={447\ndash 452},
}

\bib{brass2005research}{book}{
      author={Brass, Peter},
      author={Moser, William~OJ},
      author={Pach, J{\'a}nos},
       title={Research problems in discrete geometry},
   publisher={Springer},
        date={2005},
      volume={18},
}

\bib{CP22}{article}{
      author={Cameron, Alex},
      author={Puleo, Gregory~J},
       title={A lower bound on the saturation number, and graphs for which it is sharp},
        date={2022},
     journal={Discrete Mathematics},
      volume={345},
      number={7},
       pages={112867},
}

\bib{erdos1964problem}{article}{
      author={Erdos, Paul},
      author={Hajnal, Andr{\'a}s},
      author={Moon, John~W},
       title={A problem in graph theory},
        date={1964},
     journal={The American Mathematical Monthly},
      volume={71},
      number={10},
       pages={1107\ndash 1110},
}

\bib{faudree2011survey}{article}{
      author={Faudree, Jill~R},
      author={Faudree, Ralph~J},
      author={Schmitt, John~R},
       title={A survey of minimum saturated graphs},
        date={2011},
     journal={The Electronic Journal of Combinatorics},
      volume={1000},
       pages={DS19\ndash Jul},
}

\bib{FG13}{article}{
      author={Faudree, Ralph~J},
      author={Gould, Ronald~J},
       title={Saturation numbers for nearly complete graphs},
        date={2013},
     journal={Graphs and Combinatorics},
      volume={29},
      number={3},
       pages={429\ndash 448},
}

\bib{fulek2013topological}{inproceedings}{
      author={Fulek, Radoslav},
      author={Ruiz-Vargas, Andres~J},
       title={Topological graphs: empty triangles and disjoint matchings},
        date={2013},
   booktitle={Proceedings of the 29th {A}nnual {S}ymposium on {C}omputational {G}eometry},
       pages={259\ndash 266},
}

\bib{furedi2013history}{incollection}{
      author={F{\"u}redi, Zolt{\'a}n},
      author={Simonovits, Mikl{\'o}s},
       title={The history of degenerate (bipartite) extremal graph problems},
        date={2013},
   booktitle={Erd{\H{o}}s {C}entennial},
   publisher={Springer},
       pages={169\ndash 264},
}

\bib{kaszonyi1986saturated}{article}{
      author={K{\'a}szonyi, L{\'a}szl{\'o}},
      author={Tuza, Zs},
       title={Saturated graphs with minimal number of edges},
        date={1986},
     journal={Journal of Graph Theory},
      volume={10},
      number={2},
       pages={203\ndash 210},
}

\bib{KPRT15}{article}{
      author={Kyn\v{c}l, Jan},
      author={Pach, J\'{a}nos},
      author={Radoi\v{c}i\'{c}, Rado\v{s}},
      author={T\'{o}th, G\'{e}za},
       title={Saturated simple and $k$-simple topological graphs},
        date={2015},
     journal={Computational Geometry: Theory and Applications},
      volume={48},
       pages={295\ndash 310},
}

\bib{pach2003unavoidable}{article}{
      author={Pach, J{\'a}nos},
      author={Solymosi, J{\'o}zsef},
      author={T{\'o}th, G{\'e}za},
       title={Unavoidable configurations in complete topological graphs},
        date={2003},
     journal={Discrete \& Computational Geometry},
      volume={30},
       pages={311\ndash 320},
}

\bib{pach2003disjoint}{inproceedings}{
      author={Pach, J{\'a}nos},
      author={T{\'o}th, G{\'e}za},
       title={Disjoint edges in topological graphs},
organization={Springer},
        date={2003},
   booktitle={Indonesia-{J}apan {J}oint {C}onference on {C}ombinatorial {G}eometry and {G}raph {T}heory},
       pages={133\ndash 140},
}

\bib{suk2012disjoint}{inproceedings}{
      author={Suk, Andrew},
       title={Disjoint edges in complete topological graphs},
        date={2012},
   booktitle={Proceedings of the 28th {A}nnual {S}ymposium on {C}omputational {G}eometry},
       pages={383\ndash 386},
}

\bib{turan1941external}{article}{
      author={Tur{\'a}n, Paul},
       title={On an external problem in graph theory},
        date={1941},
     journal={Mat. Fiz. Lapok},
      volume={48},
       pages={436\ndash 452},
}

\bib{zykov1949some}{article}{
      author={Zykov, Alexander~Aleksandrovich},
       title={On some properties of linear complexes},
        date={1949},
     journal={Matematicheskii Sbornik},
      volume={66},
      number={2},
       pages={163\ndash 188},
}

\end{biblist}
\end{bibdiv}
\end{document}